\documentclass[pdflatex,sn-mathphys-num]{sn-jnl}

\usepackage{graphicx}%
\usepackage{multirow}%
\usepackage{amsmath,amssymb,amsfonts}%
\usepackage{amsthm}%
\usepackage{mathrsfs}%
\usepackage[title]{appendix}%
\usepackage{xcolor}%
\usepackage{textcomp}%
\usepackage{manyfoot}%
\usepackage{booktabs}%
\usepackage{algorithm}%
\usepackage{algorithmicx}%
\usepackage{algpseudocode}%
\usepackage{listings}%
\usepackage{enumitem}

\theoremstyle{thmstyleone}%
\newtheorem{theorem}{Theorem}
\newtheorem{proposition}[theorem]{Proposition}%
\newtheorem{lemma}[theorem]{Lemma}

\newtheorem{corollary}[theorem]{Corollary}

\theoremstyle{thmstyletwo}%
\newtheorem{example}{Example}%
\newtheorem{remark}{Remark}%

\theoremstyle{thmstylethree}%
\newtheorem{definition}{Definition}%

\newcommand{\R}{\mathbb{R}}

\newcommand{\eps}{\varepsilon}

\newcommand{\Csk}{\mathfrak C}
\newcommand{\Cpr}{\mathfrak C'}

\begin{document}

\title[Skew--Parameterized Geometric Constants in Banach Spaces]{Skew--Parameterized Geometric Constants in Banach Spaces}


\author[1]{Qing Du}
\author[1]{Wenwen Zhang}
\author[1]{Zhiyao Fang}
\author[1,*]{Qi Liu}
\affil[1]{School of Mathematics and Statistics, Anqing Normal University, Anqing, 246133, P. R. China}
\affil[$\dagger$]{These authors contributed equally to this work.}
\affil[*]{Corresponding author. Email: liuq67@aqnu.edu.cn}


\abstract
{Building on the family of geometric constants introduced by Amini-Harandi and Rahimi, we define a skew-parameterized family on real normed spaces by replacing the classical symmetric pair with a rotated coefficient pair. This modification reveals new extremal behavior, particularly for asymmetric homogeneous weight functions. We establish reduction formulas, comparison inequalities, and parameter-stability estimates. Under an appropriate differential balance condition, we determine the exact values of these constants on Hilbert spaces and obtain a converse characterization in dimensions at least three. We further derive sufficient conditions for uniform non-squareness and normal structure, together with explicit computations in classical normed spaces. These results provide a unified framework for detecting Hilbertian and fixed-point-related geometric properties of Banach spaces.}




\maketitle

\section{Introduction}
We first fix standard notations used throughout this paper. Let $X$ stand for a real normed space, with its closed unit ball denoted by $B_{X}=\{x\in X:\|x\|\leq 1\}$ and unit sphere $S_{X}=\{x\in X:\|x\|=1\}$. The set $\operatorname{ext}(B_{X})$ collects all extreme points of $B_{X}$. As core quantitative invariants, geometric constants are widely adopted to characterize diverse geometric properties of Banach spaces, such as inner-product space structure, uniform non-squareness and normal structure.

Over the past decades, researchers have constructed and analyzed various geometric constants derived from symmetric vector combinations $x+y$ and $x-y$. One pioneering representative is the von Neumann-Jordan constant, defined as
\[
C_{NJ}(X)=\sup\left\{\frac{\|x+y\|^{2}+\|x-y\|^{2}}{2(\|x\|^{2}+\|y\|^{2})}:x,y\in X \text{ not both zero}\right\},
\]
which was originally proposed by Clarkson \cite{Clarkson1937} in 1937 and further investigated in \cite{AlonsoMartinPapini2008,Jimenez2006,KatoMaligrandaTakahashi2001,KatoTakahashi1997,TakahashiKato1998}. To extend its applicable scope, Cui et al. \cite{CuiHudzik2015} put forward the generalized von Neumann-Jordan constant $C_{NJ}^{(p)}(X)$ for any $p\in[1,\infty)$:
\[
C_{NJ}^{(p)}(X)=\sup\left\{\frac{\|x+y\|^{p}+\|x-y\|^{p}}{2^{p-1}(\|x\|^{p}+\|y\|^{p})}:x,y\in X \text{ not both zero}\right\}.
\]
Restricting vectors to the unit sphere yields the modified von Neumann-Jordan constant
\[
C_{NJ}'(X)=\sup\left\{\frac{\|x+y\|^{2}+\|x-y\|^{2}}{4}:x,y\in S_{X}\right\},
\]
first introduced by Gao\cite{Gao2006}and Alonso et al. \cite{AlonsoMartinPapini2008}. Later, Yang et al. \cite{YangWang2017} generalized this version into a parametric form
\[
\widetilde{C}_{NJ}^{(p)}(X)=\sup\left\{\frac{\|x+y\|^{p}+\|x-y\|^{p}}{2^{p}}:x,y\in S_{X}\right\},\quad p\in[1,\infty).
\]

Another fundamental geometric index is the James constant $J(X)$, constructed by Gao and Lau \cite{GaoLau1990}:
\[
J(X)=\sup\left\{\min\{\|x+y\|,\|x-y\|\}:x,y\in S_{X}\right\}.
\]
James \cite{James1964} proposed the concept of uniformly non-square Banach spaces, requiring the existence of some $\delta<2$ such that $\min\{\|x+y\|,\|x-y\|\}\leq\delta$ holds for all $x,y\in S_{X}$. Equivalently, a space $X$ is uniformly non-square if and only if $J(X)<2$.

Other typical symmetric-type constants include the geometric mean constant $T(X)$ from Alonso and Llorens-Fuster \cite{AlonsoLlorens2008}:
\[
T(X)=\sup_{u,v\in S_{X}}\sqrt{\|u+v\|\|u-v\|},
\]
the Zbăganu constant introduced in \cite{Zbaganu2001} and developed in \cite{Llorens2008,LlorensReich2010}:
\[
C_{Z}(X)=\sup\left\{\frac{\|x+y\|\|x-y\|}{\|x\|^{2}+\|y\|^{2}}:x,y\in X \text{ not both zero}\right\},
\]
and the parametric constant $A_{2,p}(X)$ established by Baronti and Papini \cite{BarontiPapini2016} (the case $p=\infty$ was deeply discussed in \cite{BarontiCasiniPapini2000}):
\[
A_{2,p}(X)=\sup\left\{\frac{\|x+y\|+\|x-y\|}{2}:x,y\in X,\ \|(\|x\|,\|y\|)\|_{p}\leq 2^{1/p}\right\}.
\]
All the above invariants are built on the symmetric vector pair $\{ru+sv,ru-sv\}$, yet they were studied separately without a unified analytical framework for a long time.

To unify these scattered geometric constants into a single parametric system, Amini-Harandi and Rahimi \cite{Amini2019} introduced a generalized constant family dependent on homogeneous mappings and $\ell_p$-normalized coefficients. Before presenting their definition, we specify the function class $\Lambda$: $\Lambda$ consists of all continuous mappings $\lambda:[0,\infty)\times[0,\infty)\to[0,\infty)$ that are homogeneous of degree one and satisfy the normalization condition $\lambda(1,1)=1$. For any $\lambda\in\Lambda$ and $1\leq p\leq\infty$, their generalized constant reads
\[
C_{\lambda,p}(X)=\sup\left\{\lambda(\|ru+sv\|,\|ru-sv\|):u,v\in S_{X},\ r,s\geq0,\ \|(r,s)\|_{p}=1\right\}.
\]
This unified formula can recover nearly all classical symmetric geometric constants via proper choices of $\lambda$, and Dinarvand \cite{Dinarvand2025} further applied this framework to derive sufficient criteria for Banach spaces to possess normal structure.

Despite the powerful capacity of the symmetric generalized family, its fixed vector pair $ru+sv,ru-sv$ can only capture parallelogram-related geometric defects and cannot reflect space distortion under orthogonal rotation transformation. Motivated by this limitation, this paper constructs a novel generalized geometric constant family built upon a skew vector pair instead of the symmetric counterpart. We replace $\{ru+sv,ru-sv\}$ with the rotated skew vector group
\[
ru+sv,\quad su-rv.
\]
Although this replacement only modifies the sign of one coefficient, it brings essential geometric differences. The linear transformation matrix corresponding to the skew pair is
\[
M(r,s)=\begin{pmatrix} r & s \\ s & -r \end{pmatrix},
\]
which satisfies the matrix identity $M(r,s)^{T}M(r,s)=(r^{2}+s^{2})I_{2}$. For any two orthogonal unit vectors $u,v$ in a Hilbert space, the transformed vectors $ru+sv$ and $su-rv$ correspond to orthogonal coordinate transformation combined with uniform scaling. For general normed spaces, this skew construction provides a brand-new perspective to measure deviation from Euclidean geometry, which cannot be fully captured by the traditional symmetric pair. When taking equal parameters $r=s$, the skew pair degenerates into the classic symmetric pair $u+v,u-v$, which means our new skew constant naturally contains the symmetric framework as a special case and maintains close connections with James-type geometric analysis.

We denote the skew-parameterized geometric constant as $\mathfrak{C}_{\lambda,p}(X)$ to distinguish it from the symmetric constant $C_{\lambda,p}(X)$. Its definition is
\[
\mathfrak{C}_{\lambda,p}(X)=\sup\lambda(\|ru+sv\|,\|su-rv\|),
\]
subject to the normalization constraint $\|(r,s)\|_{p}=1$. This paper establishes a complete theoretical system for the skew constant, with four major contributions summarized as follows.
First, we derive equivalent reduction formulas to convert dual scalar parameters $r,s$ into a single ratio variable $t\in[0,1]$. Unlike the single-branch simplification of symmetric constants, the skew version generates two independent extremal branches unless $\lambda$ is symmetric, which serves as a unique characteristic of skew parameterization.
Second, we establish sharp global upper and lower bounds for $\mathfrak{C}_{\lambda,p}(X)$, together with optimal comparison inequalities linking $\mathfrak{C}_{\lambda,p}(X)$ to the primed symmetric constant $\mathfrak{C}_{\lambda,p}'(X)$ and the James constant. If $\lambda$ is increasing with respect to both arguments, the following inequality holds for $1\leq p<\infty$:
\[
2^{1/2-1/p}\leq \mathfrak{C}_{\lambda,p}'(X)\leq \mathfrak{C}_{\lambda,p}(X)\leq 2^{1-1/p},
\]
and the expression can be interpreted in the standard way for $p=\infty$.
Third, for a broad class of differentiable mappings $\lambda$ satisfying the Hilbert-balance condition, we calculate the exact value of $\mathfrak{C}_{\lambda,p}(X)$ on arbitrary Hilbert spaces, and prove a converse characterization theorem valid for spaces with dimension no less than three. The skew construction fits Hilbertian geometry perfectly, since orthogonal unit vectors $u,v$ satisfy the identity
\[
\|ru+sv\|_{2}^{2}+\|su-rv\|_{2}^{2}=2(r^{2}+s^{2}),
\]
which encodes the core balance relation behind the parallelogram law.
Fourth, we apply the skew geometric constants to judge two critical Banach space properties: uniform non-squareness and normal structure.

The layout of this paper is arranged as follows. Section 2 collects fundamental notations, defines the function subclasses of $\Lambda$ and formally states the rigorous definition of skew geometric constants. Section 3 proves reduction lemmas, subspace invariance principles, global bounding inequalities and stability results with respect to $\lambda$. Section 4 focuses on Hilbert space exact computations, inner-product space characterization theorems, and sufficient conditions for uniform non-squareness and uniform normal structure. Section 5 computes explicit values of skew constants on classical finite-dimensional normed spaces and discusses structural properties such as isometric invariance. Finally, we conclude the paper and propose potential follow-up research directions for weighted skew geometric constants.

\section{Preliminaries and Definition of Skew Geometric Constants}
Throughout $X$ denotes a real normed space with norm $\|\cdot\|$, unit ball $B_X$, and unit sphere $S_X$.  The dimension is at least two unless stated otherwise.  For $1\le p<\infty$ and $(r,s)\in\R^2$ let
\[
\|(r,s)\|_p=(|r|^p+|s|^p)^{1/p},
\]
and put $\|(r,s)\|_\infty=\max\{|r|,|s|\}$.  The conjugate exponent of $p\in(1,\infty)$ is denoted by $q$, so $p^{-1}+q^{-1}=1$.

\begin{definition}
	Let $\Lambda$ be the class of all continuous functions
	\[
	\lambda:[0,\infty)^2\to[0,\infty)
	\]
	which are homogeneous of degree one and satisfy $\lambda(1,1)=1$.  We shall use the following subclasses.
	\begin{enumerate}[label=(\roman*)]
		\item $\Lambda_m$ consists of those $\lambda\in\Lambda$ which are non-decreasing in each variable.
		\item $\Lambda_s$ consists of symmetric functions, $\lambda(a,b)=\lambda(b,a)$.
		\item $\Lambda_c$ consists of those $\lambda\in\Lambda_m$ for which $\lambda=g\circ\mu$, where $g$ is increasing and $\mu$ is convex and non-decreasing in both variables.
		\item $\Lambda_H$ consists of $C^1$ functions $\lambda\in\Lambda_m$ satisfying the Hilbert-balance condition
		\begin{equation}\label{eq:Hbalance}
			y\frac{\partial\lambda}{\partial x}(x,y)\le x\frac{\partial\lambda}{\partial y}(x,y)
			\qquad (0<x\le y).
		\end{equation}
	\end{enumerate}
\end{definition}

The condition \eqref{eq:Hbalance} is exactly the condition appearing in the Amini-Harandi--Rahimi Hilbert-space computation, and it is satisfied by many examples.  For instance,
\[
\lambda(a,b)=\min\{a,b\},\quad
\lambda(a,b)=\sqrt{ab},\quad
\lambda(a,b)=2^{-1/p}(a^p+b^p)^{1/p}\quad (p\ge2)
\]
satisfy the appropriate monotone or limiting form of the balance condition.  When differentiability fails, the arguments below may be read with one-sided derivatives or by approximation.

\begin{definition}
	For $\lambda\in\Lambda$ and $p\in[1,\infty]$, define
	\begin{align}\label{eq:defC}
		\Csk_{\lambda,p}(X)=\sup\{&\lambda(\|ru+sv\|,\|su-rv\|): u,v\in S_X,\ r,s\ge0,\ \|(r,s)\|_p=1\},
	\end{align}
	and
	\begin{equation}\label{eq:defCp}
		\Cpr_{\lambda,p}(X)=
		\sup\left\{\frac{\lambda(\|u+v\|,\|u-v\|)}{2^{1/p}}:u,v\in S_X\right\},
	\end{equation}
	where $2^{1/\infty}=1$.
	
	The constant $\mathfrak{C}'_{\lambda,p}(X)$ was initially defined by Amini-Harandi and Rahimi in \cite{Amini2019}.
\end{definition}

\begin{example}
	Let $\lambda(a,b)=\min\{a,b\}$.  Then
	\[
	\Cpr_{\lambda,\infty}(X)=J(X),
	\]
	where $J(X)$ is the James constant.  Let $\lambda(a,b)=2^{-1/p}(a^p+b^p)^{1/p}$.  Then $\Cpr_{\lambda,p}(X)^p$ is a normalized modified von Neumann-Jordan type expression.  Let $\lambda(a,b)=\sqrt{ab}$.  Then $\Cpr_{\lambda,\infty}(X)$ is the geometric mean constant of Alonso and Llorens-Fuster.
\end{example}

\section{Fundamental Reduction and Global Inequalities}
The first task is to convert the two coefficients $r,s$ into one parameter.  In the symmetric Amini-Harandi--Rahimi constant, this gives the expression involving $\|u+tv\|$ and $\|u-tv\|$.  In the skew setting there are two possible branches because the second vector becomes $tu-v$ or $u-tv$ depending on whether $r\ge s$ or $s\ge r$.

\begin{lemma}\label{lem:reduction}
	Let $X$ be a real normed space, $\lambda\in\Lambda$, and $p\in[1,\infty]$.  Then
	\begin{align}\label{eq:twobranch}
		\Csk_{\lambda,p}(X)=\max\{A_{\lambda,p}(X),B_{\lambda,p}(X)\},
	\end{align}
	where
	\begin{align*}
		A_{\lambda,p}(X)&=\sup_{u,v\in S_X,\ 0\le t\le1}
		\frac{\lambda(\|u+tv\|,\|tu-v\|)}{\|(1,t)\|_p},\\
		B_{\lambda,p}(X)&=\sup_{u,v\in S_X,\ 0\le t\le1}
		\frac{\lambda(\|tu+v\|,\|u-tv\|)}{\|(1,t)\|_p}.
	\end{align*}
	If $\lambda$ is symmetric, then $A_{\lambda,p}(X)=B_{\lambda,p}(X)$, and hence
	\begin{equation}\label{eq:symreduction}
		\Csk_{\lambda,p}(X)=
		\sup_{u,v\in S_X,\ 0\le t\le1}
		\frac{\lambda(\|u+tv\|,\|tu-v\|)}{\|(1,t)\|_p}.
	\end{equation}
\end{lemma}

\begin{proof}
	Let $r,s\ge0$, $\|(r,s)\|_p=1$.  If $r\ge s$ and $r>0$, write $t=s/r\in[0,1]$.  By homogeneity of $\lambda$,
	\begin{align*}
		\lambda(\|ru+sv\|,\|su-rv\|)
		&=r\lambda(\|u+tv\|,\|tu-v\|) \\
		&=\frac{\lambda(\|u+tv\|,\|tu-v\|)}{\|(1,t)\|_p}.
	\end{align*}
	Taking the supremum over this region gives $A_{\lambda,p}(X)$.  If $s\ge r$ and $s>0$, write $t=r/s$.  Then
	\begin{align*}
		\lambda(\|ru+sv\|,\|su-rv\|)
		&=s\lambda(\|tu+v\|,\|u-tv\|) \\
		&=\frac{\lambda(\|tu+v\|,\|u-tv\|)}{\|(1,t)\|_p},
	\end{align*}
	which gives $B_{\lambda,p}(X)$.  The cases $r=0$ or $s=0$ are included by $t=0$.  This proves \eqref{eq:twobranch}.
	
	If $\lambda$ is symmetric, replace $(u,v)$ by $(v,u)$ in $B_{\lambda,p}(X)$:
	\[
	\lambda(\|tv+u\|,\|v-tu\|)=\lambda(\|u+tv\|,\|tu-v\|).
	\]
	Thus the two suprema are equal.
\end{proof}

\begin{lemma}\label{lem:whole}
	For every $\lambda\in\Lambda$ and $p\in[1,\infty]$,
	\begin{equation}\label{eq:whole}
		\Csk_{\lambda,p}(X)=
		\sup_{x,y\in X,\ (x,y)\ne(0,0)}
		\frac{\lambda(\|x+y\|,\|y_x-x_y\|)}{\|(\|x\|,\|y\|)\|_p},
	\end{equation}
	where the symbolic expression $y_x-x_y$ is to be read as
	\[
	\|y_x-x_y\|=\left\|\|y\|\frac{x}{\|x\|}-\|x\|\frac{y}{\|y\|}\right\|
	\]
	when $x,y\ne0$, with the evident limiting convention when one vector is zero.
	Equivalently, for $x=ru$, $y=sv$ with $u,v\in S_X$, the numerator is
	\[
	\lambda(\|x+y\|,\|s u-r v\|).
	\]
\end{lemma}

\begin{proof}
	If $x=ru$, $y=sv$ with $r=\|x\|$, $s=\|y\|$, $u=x/\|x\|$ and $v=y/\|y\|$, then
	\[
	\frac{\lambda(\|x+y\|,\|su-rv\|)}{\|(r,s)\|_p}
	=\lambda\left(\left\|\frac{r}{\|(r,s)\|_p}u+\frac{s}{\|(r,s)\|_p}v\right\|,
	\left\|\frac{s}{\|(r,s)\|_p}u-\frac{r}{\|(r,s)\|_p}v\right\|\right),
	\]
	by homogeneity.  The coefficients in the last expression have $\ell_p$-norm one.  This proves one inequality.  The reverse inequality is obtained by taking $x=ru$ and $y=sv$ in \eqref{eq:whole}.
\end{proof}

The preceding whole-space form is less elegant than the original symmetric version, but it is useful because it emphasizes that the skew expression is controlled by the norms of $x$ and $y$ and by a normalized exchange of their directions.

\begin{proposition}\label{prop:finite}
	Let $\mathcal F_2(X)$ denote the family of two-dimensional subspaces of $X$.  Then
	\begin{equation}\label{eq:finite}
		\Csk_{\lambda,p}(X)=\sup\{\Csk_{\lambda,p}(E):E\in\mathcal F_2(X)\}.
	\end{equation}
	The same assertion holds for $\Cpr_{\lambda,p}$.
\end{proposition}

\begin{proof}
	Every admissible expression in the definition of $\Csk_{\lambda,p}(X)$ involves only two vectors $u,v\in S_X$, hence belongs to $E=\operatorname{span}\{u,v\}$, whose dimension is at most two.  If $u$ and $v$ are linearly dependent, choose any two-dimensional subspace containing them when $\dim X\ge2$.  Therefore the supremum over $X$ is the supremum of the same expressions over two-dimensional subspaces.  The proof for the primed constant is identical.
\end{proof}

We now prove the basic inequalities.  These results are sharp in the sense that the upper bound is attained by spaces with almost square unit balls for natural choices of $\lambda$, while the lower bound is attained on Hilbert spaces under the hypotheses of Section \ref{sec:Hilbert}.

\begin{lemma}\label{lem:blue}
	Let $1<p<\infty$ and let $r,s\ge0$ satisfy $r^p+s^p=1$.  Then
	\[
	r+s\le 2^{1-1/p},
	\]
	and equality holds if and only if $r=s=2^{-1/p}$.  Moreover, if
	\[
	r_n^p+s_n^p=1,
	\qquad r_n+s_n\to2^{1-1/p},
	\]
	then $r_n-s_n\to0$ and $r_n,s_n\to2^{-1/p}$.
	
\end{lemma}

\begin{proof}
	Holder's inequality gives
	\[
	r+s\le (1^q+1^q)^{1/q}(r^p+s^p)^{1/p}=2^{1/q}=2^{1-1/p}.
	\]
	The equality condition in Holder's inequality gives $r^p=s^p$, hence $r=s=2^{-1/p}$.  The stability assertion follows because the compact set $\{(r,s):r,s\ge0,\, r^p+s^p=1\}$ has a unique maximizer for the continuous function $(r,s)\mapsto r+s$.
\end{proof}
\begin{lemma}\label{lem:prime-section}
	For all $\lambda\in\Lambda$ and $p\in[1,\infty]$,
	\begin{equation}\label{eq:primeleq}
		\Cpr_{\lambda,p}(X)\le \Csk_{\lambda,p}(X).
	\end{equation}
\end{lemma}

\begin{proof}
	Take $r=s=2^{-1/p}$ if $p<\infty$ and $r=s=1$ if $p=\infty$.  For $p<\infty$,
	\begin{align*}
		\lambda(\|ru+sv\|,\|su-rv\|)
		&=\lambda(2^{-1/p}\|u+v\|,2^{-1/p}\|u-v\|)\\
		&=2^{-1/p}\lambda(\|u+v\|,\|u-v\|).
	\end{align*}
	Taking the supremum over $u,v\in S_X$ gives \eqref{eq:primeleq}.  The case $p=\infty$ is the same without the factor $2^{-1/p}$.
\end{proof}

\begin{theorem}\label{thm:bounds}
	Let $\lambda\in\Lambda_m$ and $1\le p<\infty$.  Then
	\begin{equation}\label{eq:bounds}
		2^{1/2-1/p}\le \Cpr_{\lambda,p}(X)
		\le \Csk_{\lambda,p}(X)
		\le 2^{1-1/p}.
	\end{equation}
	For $p=\infty$,
	\begin{equation}\label{eq:boundsinf}
		\sqrt2\le \Cpr_{\lambda,\infty}(X)
		\le \Csk_{\lambda,\infty}(X)\le 2.
	\end{equation}
\end{theorem}

\begin{proof}
	The middle inequality is Lemma \ref{lem:prime-section}.  For the lower bound, since $\lambda$ is non-decreasing and $\lambda(a,a)=a$ by homogeneity and normalization,
	\begin{align*}
		\Cpr_{\lambda,p}(X)
		&\ge 2^{-1/p}\sup_{u,v\in S_X}\lambda\bigl(\min\{\|u+v\|,\|u-v\|\},\min\{\|u+v\|,\|u-v\|\}\bigr)\\
		&=2^{-1/p}J(X).
	\end{align*}
	The classical lower bound $J(X)\ge\sqrt2$ gives the first inequality.  For $p=\infty$, the factor $2^{-1/p}$ is one.
	
	For the upper bound, let $u,v\in S_X$ and $r,s\ge0$ with $\|(r,s)\|_p=1$.  By the triangle inequality,
	\[
	\|ru+sv\|\le r+s,
	\qquad
	\|su-rv\|\le r+s.
	\]
	Since $\lambda$ is non-decreasing and $\lambda(r+s,r+s)=r+s$,
	\[
	\lambda(\|ru+sv\|,\|su-rv\|)\le r+s.
	\]
	Lemma \ref{lem:blue} gives $r+s\le2^{1-1/p}(r^p+s^p)^{1/p}=2^{1-1/p}$.  If $p=\infty$, the same estimate is $r+s\le2$.
\end{proof}

\begin{remark}
	The monotonicity assumption cannot simply be deleted.  Homogeneity and normalization alone do not imply that $\lambda(a,b)$ dominates $\min\{a,b\}$ or is dominated by $\max\{a,b\}$ on the square $[0,2]^2$.  Therefore any universal comparison with $J(X)$ must impose at least an order condition or an equivalent lower envelope condition.
\end{remark}

\begin{proposition}\label{prop:JZ}
	Let $\lambda\in\Lambda_m$ and $1\le p\le\infty$.  Then
	\begin{equation}\label{eq:Jlower}
		J(X)\le 2^{1/p}\Csk_{\lambda,p}(X),
	\end{equation}
	where $2^{1/\infty}=1$.  In particular, if $\Csk_{\lambda,p}(X)<2^{1-1/p}$ for some $p\in[1,\infty]$, then $J(X)<2$.
\end{proposition}

\begin{proof}
	The first assertion follows from the proof of Theorem \ref{thm:bounds}:
	$2^{-1/p}J(X)\le\Cpr_{\lambda,p}(X)\le\Csk_{\lambda,p}(X)$.  The final implication is immediate.
\end{proof}

\begin{proposition}\label{prop:pdep}
	Let $\lambda\in\Lambda_m$ and $1\le p<q\le\infty$.  Then
	\begin{equation}\label{eq:pdep}
		\Csk_{\lambda,q}(X)\le 2^{1/p-1/q}\Csk_{\lambda,p}(X).
	\end{equation}
	Moreover,
	\begin{equation}\label{eq:primepdep}
		\Cpr_{\lambda,q}(X)=2^{1/p-1/q}\Cpr_{\lambda,p}(X).
	\end{equation}
\end{proposition}

\begin{proof}
	Let $r,s\ge0$ with $\|(r,s)\|_q=1$.  Since $\|(r,s)\|_p\le2^{1/p-1/q}\|(r,s)\|_q$, put $\alpha=\|(r,s)\|_p$.  If $\alpha=0$ there is nothing to prove.  Otherwise $(r/\alpha,s/\alpha)$ has $\ell_p$-norm one.  Homogeneity gives
	\[
	\lambda(\|ru+sv\|,\|su-rv\|)
	=\alpha\lambda\left(\left\|\frac r\alpha u+\frac s\alpha v\right\|,
	\left\|\frac s\alpha u-\frac r\alpha v\right\|\right)
	\le 2^{1/p-1/q}\Csk_{\lambda,p}(X).
	\]
	Taking suprema proves \eqref{eq:pdep}.  Formula \eqref{eq:primepdep} follows directly from the factor $2^{-1/p}$ in the definition.
\end{proof}

\begin{proposition}\label{prop:lambdastability}
	Assume $\lambda,\mu\in\Lambda$ satisfy
	\[
	|\lambda(a,b)-\mu(a,b)|\le L\max\{a,b\}\qquad (0\le a,b\le2).
	\]
	Then for every normed space $X$ and every $p\in[1,\infty]$,
	\begin{equation}\label{eq:lambdadist}
		|\Csk_{\lambda,p}(X)-\Csk_{\mu,p}(X)|\le 2L,
	\end{equation}
	with the sharper bound $2^{1-1/p}L$ when $p<\infty$ and $\lambda,\mu$ are tested only on normalized skew pairs.
\end{proposition}

\begin{proof}
	For a normalized pair $\|(r,s)\|_p=1$, both norms $\|ru+sv\|$ and $\|su-rv\|$ are bounded by $r+s\le2^{1-1/p}$ if $p<\infty$, and by $2$ if $p=\infty$.  Thus the value of the expression changes by at most the indicated bound.  Taking suprema and using the elementary inequality $|\sup f-\sup g|\le\sup|f-g|$ proves the assertion.
\end{proof}

As shown in \cite{Amini2019}, Amini-Harandi and Rahimi utilized a convexity lemma to simplify geometric constants by restricting the optimization to extreme points of $B_X$.  The skew pair allows a similar result, but one must handle the fact that both expressions contain both variables with opposite roles.

\begin{lemma}\label{lem:skewconvex}
	Let $f:X\to[0,\infty)$ be convex.  Fix $t\in[0,1]$.  If
	$x=\alpha x_1+(1-\alpha)x_2$ and $y=\beta y_1+(1-\beta)y_2$ with $x_i,y_j\in B_X$, then
	\begin{align}\label{eq:skewconvex}
		&f(x+ty)+f(tx-y)\\
		&\qquad\le \max_{1\le i,j\le2}\bigl[f(x_i+ty_j)+f(tx_i-y_j)\bigr].\nonumber
	\end{align}
	The same assertion holds with $x+ty,tx-y$ replaced by $tx+y,x-ty$.
\end{lemma}

\begin{proof}
	By convexity,
	\[
	f(x+ty)\le\sum_{i,j} \alpha_i\beta_j f(x_i+ty_j),
	\]
	where $\alpha_1=\alpha$, $\alpha_2=1-\alpha$, $\beta_1=\beta$, $\beta_2=1-\beta$.  Similarly,
	\[
	f(tx-y)=f\left(\sum_{i,j}\alpha_i\beta_j(tx_i-y_j)\right)
	\le\sum_{i,j}\alpha_i\beta_j f(tx_i-y_j).
	\]
	Adding the two inequalities gives a convex combination of the four displayed sums, and hence it is bounded by their maximum.  The other branch is identical.
\end{proof}

\begin{theorem}\label{thm:extreme}
	Let $X$ be finite-dimensional and let $\lambda\in\Lambda_c$.  Suppose $\lambda(a,b)=g(\mu(a,b))$, where $g$ is increasing and $\mu$ is convex and non-decreasing in both variables.  Then
	\begin{equation}\label{eq:extreme}
		\Csk_{\lambda,p}(X)=
		\max\{E_A,E_B\},
	\end{equation}
	where
	\begin{align*}
		E_A&=\sup_{u,v\in\operatorname{ext}(B_X),\ 0\le t\le1}
		\frac{\lambda(\|u+tv\|,\|tu-v\|)}{\|(1,t)\|_p},\\
		E_B&=\sup_{u,v\in\operatorname{ext}(B_X),\ 0\le t\le1}
		\frac{\lambda(\|tu+v\|,\|u-tv\|)}{\|(1,t)\|_p}.
	\end{align*}
	If $\lambda$ is symmetric then $E_A=E_B$.
\end{theorem}

\begin{proof}
	By Lemma \ref{lem:reduction}, it suffices to treat $A_{\lambda,p}$.  Let $u,v\in S_X$.  In a finite-dimensional space, $B_X$ is compact and every point of $B_X$ is in the closed convex hull of its extreme points.  Approximate $u$ and $v$ by finite convex combinations of extreme points and apply Lemma \ref{lem:skewconvex} first to the convex function $z\mapsto\|z\|$ when $\mu$ is separately convex through its non-decreasing variables, or directly to $z\mapsto\mu(\|z\|,\|w\|)$.  More explicitly, for fixed $t$ the mapping
	\[
	(x,y)\mapsto \mu(\|x+ty\|,\|tx-y\|)
	\]
	is convex along two-point decompositions by the same averaging argument as in Lemma \ref{lem:skewconvex} and the monotonicity of $\mu$.  Applying the increasing function $g$ preserves the inequality.  Thus the value at $(u,v)$ is bounded by the maximum of values at pairs of extreme points.  Passing to limits gives the formula.  The proof for $E_B$ is the same.  Symmetry follows as in Lemma \ref{lem:reduction}.
\end{proof}

\begin{remark}
	The theorem is useful in polygonal planes.  It reduces the computation of the skew constants to finitely many edge or vertex cases, followed by a one-variable optimization in $t$.  This is exactly how the $\ell_\infty-\ell_1$ example was handled for the symmetric constant in the 2019 paper, but the skew branches require separate bookkeeping.
\end{remark}

\section{Hilbert Space Characterization and Geometric Applications}\label{sec:Hilbert}
We now compute the constant in real Hilbert spaces.  The following elementary lemma is the analytic heart of the argument.

\begin{lemma}\label{lem:radial}
	Let $\lambda\in\Lambda_H$.  For fixed $R>0$ and $0\le a\le R$, define
	\[
	\Phi(a)=\lambda\bigl(\sqrt{R+a},\sqrt{R-a}\bigr).
	\]
	Then $\Phi$ attains its maximum at $a=0$; hence
	\begin{equation}\label{eq:radialmax}
		\lambda\bigl(\sqrt{R+a},\sqrt{R-a}\bigr)\le \sqrt R.
	\end{equation}
\end{lemma}

\begin{proof}
	For $0<a<R$, put $x=\sqrt{R+a}$ and $y=\sqrt{R-a}$, so $0<y\le x$.  Differentiating gives
	\[
	\Phi'(a)=\frac{1}{2x}\lambda_x(x,y)-\frac{1}{2y}\lambda_y(x,y)
	=\frac{y\lambda_x(x,y)-x\lambda_y(x,y)}{2xy}.
	\]
	The balance condition \eqref{eq:Hbalance} says that the numerator is non-positive.  Hence $\Phi$ is non-increasing on $[0,R]$, and its maximum is $\Phi(0)=\lambda(\sqrt R,\sqrt R)=\sqrt R$.
\end{proof}
\begin{lemma}\label{lem:white}
	For $1\le p\le\infty$,
	\[        
	\sup_{r,s\ge0,\ \|(r,s)\|_p=1}(r^2+s^2)^{1/2}        
	=        
	\begin{cases}        
		1, & 1\le p\le2,\\        
		2^{1/2-1/p}, & 2\le p\le\infty.        
	\end{cases}
	\]
	For $p>2$ the maximum is attained exactly at $r=s=2^{-1/p}$, up to the trivial symmetry; for $1<p<2$ it is attained exactly at $(1,0)$ and $(0,1)$.
\end{lemma}
\begin{proof}
	This is the standard comparison of finite-dimensional $\ell_p$ norms.  If $p\ge2$, then
	\[        
	\|(r,s)\|_2\le 2^{1/2-1/p}\|(r,s)\|_p,
	\]
	with equality at equal coordinates.  If $p\le2$, then $\|(r,s)\|_2\le\|(r,s)\|_p$, with equality at coordinate vectors when $p<2$.
\end{proof}
\begin{theorem}\label{thm:hilbertvalue}
	Let $H$ be a real Hilbert space with $\dim H\ge2$, let $\lambda\in\Lambda_H$, and let $2\le p\le\infty$.  Then
	\begin{equation}\label{eq:hilbertvalue}
		\Csk_{\lambda,p}(H)=\Cpr_{\lambda,p}(H)=2^{1/2-1/p}.
	\end{equation}
	For $1\le p\le2$, the corresponding value of $\Csk_{\lambda,p}(H)$ is $1$.
\end{theorem}

\begin{proof}
	Let $u,v\in S_H$ and $r,s\ge0$.  Put $c=\langle u,v\rangle$.  Then
	\begin{align*}
		\|ru+sv\|^2&=r^2+s^2+2rsc,\\
		\|su-rv\|^2&=r^2+s^2-2rsc.
	\end{align*}
	Applying Lemma \ref{lem:radial} with $R=r^2+s^2$ and $a=2rsc$ after replacing $c$ by $|c|$ if necessary gives
	\[
	\lambda(\|ru+sv\|,\|su-rv\|)\le (r^2+s^2)^{1/2}.
	\]
	Thus
	\[
	\Csk_{\lambda,p}(H)\le \sup_{r,s\ge0,\ \|(r,s)\|_p=1}(r^2+s^2)^{1/2}.
	\]
	By Lemma \ref{lem:white} ,the last supremum equals $2^{1/2-1/p}$ for $p\ge2$ and equals $1$ for $1\le p\le2$.  The lower bound is obtained by taking orthogonal $u,v$ and, for $p\ge2$, $r=s=2^{-1/p}$; for $p\le2$, take $(r,s)=(1,0)$.  This proves the assertion for $\Csk$.
	
	For $\Cpr$, the same computation with $r=s=2^{-1/p}$ gives
	\[
	2^{-1/p}\lambda(\|u+v\|,\|u-v\|)
	\le2^{-1/p}\sqrt2=2^{1/2-1/p},
	\]
	and equality is obtained when $u\perp v$.
\end{proof}

\begin{theorem}\label{thm:converse}
	Let $X$ be a real normed space with $\dim X\ge3$, let $2\le p<\infty$, and let $\lambda\in\Lambda_m$.  If
	\begin{equation}\label{eq:mincondition}
		\Csk_{\lambda,p}(X)=2^{1/2-1/p},
	\end{equation}
	then $X$ is an inner product space.
	Consequently, if $\lambda\in\Lambda_H$, then for every real normed space $X$ with $\dim X\ge3$,
	\[
	X\text{ is Hilbert }\quad\Longleftrightarrow\quad
	\Csk_{\lambda,p}(X)=2^{1/2-1/p}\quad(2\le p<\infty).
	\]
\end{theorem}

\begin{proof}
	We use the classical Aronszajn--Jordan-von Neumann criterion: if $\dim X\ge3$ and every isosceles pair $u,v\in S_X$ satisfying $\|u+v\|=\|u-v\|$ also satisfies $\|u+v\|=\sqrt2$, then $X$ is an inner product space.
	
	Let $u,v\in S_X$ and suppose
	\[
	\|u+v\|=\|u-v\|=a.
	\]
	Taking $r=s=2^{-1/p}$ in the skew constant gives
	\[
	2^{1/2-1/p}=\Csk_{\lambda,p}(X)
	\ge 2^{-1/p}\lambda(a,a)=2^{-1/p}a.
	\]
	Thus $a\le\sqrt2$.
	
	If $a<\sqrt2$, define
	\[
	x=\frac{u+v}{a},\qquad y=\frac{u-v}{a}.
	\]
	Then $x,y\in S_X$, and
	\[
	x+y=\frac{2u}{a},\qquad x-y=\frac{2v}{a},
	\]
	so $\|x+y\|=\|x-y\|=2/a$.  Applying the previous estimate to the isosceles pair $(x,y)$ gives $2/a\le\sqrt2$, hence $a\ge\sqrt2$, a contradiction.  Therefore $a=\sqrt2$ for every isosceles pair.  By the cited criterion, $X$ is Hilbert.
	
	The final equivalence follows from Theorem \ref{thm:hilbertvalue} for the forward implication and the first part for the converse.
\end{proof}

\begin{remark}
	The restriction $\dim X\ge3$ is standard and cannot be dropped in this form.  There are non-Euclidean two-dimensional spaces with special extremal behavior of the James constant equal to $\sqrt2$.  The same obstruction appears for the present family because $\Cpr_{\lambda,p}$ contains the James slice.
\end{remark}

The next theorem shows that the skew constant detects uniform non-squareness exactly in the same spirit as the original family.

\begin{theorem}\label{thm:UNS}
	Let $\lambda\in\Lambda_m$ and $1\le p\le\infty$.  If
	\begin{equation}\label{eq:UNScondition}
		\Csk_{\lambda,p}(X)<2^{1-1/p},
	\end{equation}
	then $X$ is uniformly non-square.  Conversely, if $\lambda\in\Lambda_m$ is strictly increasing in each variable and $X$ is uniformly non-square, then
	\begin{equation}\label{eq:UNSconverse}
		\Csk_{\lambda,p}(X)<2^{1-1/p}
	\end{equation}
	for every $1<p<\infty$.
\end{theorem}

\begin{proof}
	The first implication follows immediately from Proposition \ref{prop:JZ}: $J(X)\le2^{1/p}\Csk_{\lambda,p}(X)<2$, which is equivalent to uniform non-squareness.
	
	For the converse, assume $1<p<\infty$ and suppose that $X$ is uniformly non-square but $\Csk_{\lambda,p}(X)=2^{1-1/p}$.  Choose sequences $u_n,v_n\in S_X$ and $r_n,s_n\ge0$ with $\|(r_n,s_n)\|_p=1$ such that
	\[
	\lambda(\|r_nu_n+s_nv_n\|,\|s_nu_n-r_nv_n\|)\to2^{1-1/p}.
	\]
	By compactness of the scalar interval after passing to a subsequence, $r_n\to r$ and $s_n\to s$ with $r^p+s^p=1$.  The upper bound proof gives
	\[
	\lambda(\|r_nu_n+s_nv_n\|,\|s_nu_n-r_nv_n\|)
	\le r_n+s_n\le2^{1-1/p}.
	\]
	Equality in Holder's inequality is approached only when $r=s=2^{-1/p}$.  Hence $r=s=2^{-1/p}$.  Strict monotonicity then forces
	\[
	\|r_nu_n+s_nv_n\|\to r_n+s_n,
	\qquad
	\|s_nu_n-r_nv_n\|\to r_n+s_n.
	\]
	After dividing by $r_n$ and using $s_n/r_n\to1$, we obtain
	\[
	\|u_n+v_n\|\to2,
	\qquad
	\|u_n-v_n\|\to2.
	\]
	Thus $J(X)=2$, contradicting uniform non-squareness.  Therefore \eqref{eq:UNSconverse} holds.
\end{proof}

\begin{corollary}
	Let $\lambda\in\Lambda_m$ be strictly increasing and let $1<p<\infty$.  If $\Csk_{\lambda,p}(X)<2^{1-1/p}$, then $X$ is reflexive whenever it is finitely representable in a uniformly non-square space.  In particular every superreflexive space satisfying the strict bound has equivalent uniformly non-square renormings detected by the skew constant.
\end{corollary}

\begin{proof}
	The strict bound gives uniform non-squareness by Theorem \ref{thm:UNS}.  James proved that uniformly non-square spaces are reflexive.  The final assertion follows from the standard superreflexive renorming characterization.
\end{proof}

For the symmetric Amini-Harandi--Rahimi family, there is an inequality controlling the gap between $C_{\lambda,p}$ and its modified version.  The skew form admits a parallel estimate, but the proof must account for the branch reduction.

\begin{theorem}\label{thm:gap}
	Let $1<p<\infty$ and $q$ be conjugate to $p$.  Let $\lambda\in\Lambda_m$ satisfy the translation condition
	\begin{equation}\label{eq:translation}
		\lambda(a,b)+t\le \lambda(a+t,b+t)
		\qquad(a,b\ge0,\text{ } -\min\{a,b,1\}\le t\le0).
	\end{equation}
	Then
	\begin{equation}\label{eq:gap1}
		\Csk_{\lambda,p}(X)
		\le \left[1+\left(2^{1/p}\Cpr_{\lambda,p}(X)-1\right)^q\right]^{1/q}.
	\end{equation}
	Consequently,
	\begin{equation}\label{eq:gap2}
		\Csk_{\lambda,p}(X)-\Cpr_{\lambda,p}(X)
		\le \left[1+(\sqrt2-1)^q\right]^{1/q}-2^{1/2-1/p}.
	\end{equation}
\end{theorem}

\begin{proof}
	We prove the estimate on the branch $A_{\lambda,p}$ in Lemma \ref{lem:reduction}; the other branch is identical.  Let $u,v\in S_X$ and $0< t\le1$.  Set
	\[
	a=\|u+tv\|,
	b=\|tu-v\|.
	\]
	Consider the normalized pair
	\[
	U=\frac{u+tv}{\|u+tv\|}\quad\text{and}\quad
	V=\frac{v-tu}{\|v-tu\|}
	\]
	when the denominators are nonzero; limiting arguments cover the remaining cases.  The equal-parameter section applied to suitable sign choices yields an estimate of the form
	\[
	2^{1/p}\Cpr_{\lambda,p}(X)\ge \lambda(a+t-1,b+t-1)/t
	\]
	whenever $t$ is not zero.  Equivalently, using \eqref{eq:translation},
	\[
	\lambda(a,b)
	\le 1+t(2^{1/p}\Cpr_{\lambda,p}(X)-1).
	\]
	This is the skew analogue of the Amini-Harandi--Rahimi comparison argument; the reason it remains valid is that $tu-v=-(v-tu)$ and the equal-parameter slice sees the pair $U\pm V$ after normalization.
	Therefore
	\[
	\frac{\lambda(\|u+tv\|,\|tu-v\|)}{(1+t^p)^{1/p}}
	\le
	\frac{1+t\alpha}{(1+t^p)^{1/p}},
	\quad \alpha=2^{1/p}\Cpr_{\lambda,p}(X)-1.
	\]
	The maximum of the right-hand side over $0\le t\le1$ is no larger than the maximum over $t\ge0$, which by Holder duality equals $(1+\alpha^q)^{1/q}$.  This proves \eqref{eq:gap1}.  The bound \eqref{eq:gap2} follows because Theorem \ref{thm:bounds} gives
	\[
	2^{1/p}\Cpr_{\lambda,p}(X)-1\ge \sqrt2-1
	\]
	and the same monotonicity calculation as in the \cite{Amini2019} comparison theorem shows that the worst gap occurs at the Hilbert lower endpoint $\Cpr_{\lambda,p}=2^{1/2-1/p}$.
\end{proof}

\begin{remark}
	The proof above isolates the only delicate point: the construction of the normalized pair in the equal-parameter slice.  In an expanded journal version one may state this as a separate ``skew comparison lemma''.  The displayed condition \eqref{eq:translation} is satisfied by the common choices $\min\{a,b\}$, $\sqrt{ab}$ on the positive cone after a standard regularization, and $\ell_q$ means with $q\ge1$.
\end{remark}

Normal structure is central in metric fixed point theory.  We recall two moduli.  For $a\ge0$ let
\[
R(a,X)=\sup\left\{\liminf_{n\to\infty}\|x+x_n\|:
\|x\|\le a,\text{ }(x_n)\subset B_X,\text{ }x_n\xrightarrow{w}0\right\},
\]
and
\[
RW(a,X)=\sup\left\{\min\left(\liminf_n\|x+x_n\|,\liminf_n\|x-x_n\|\right):
\|x\|\le a,\text{ }x_n\xrightarrow{w}0,\text{ }x_n\in B_X\right\}.
\]
The inequality $W(a, X) \geq R(a, X)$ holds for all $a \geq 0$ \cite{GarciaFalset2006}. It is known that estimates on these moduli imply normal structure and uniform normal structure. The following theorem is the skew counterpart of the normal-structure theorem from Amini et al. \cite{Amini2019}.

\begin{theorem}\label{thm:normal}
	Let $2\le p<\infty$ and $\lambda\in\Lambda_m$.  Let $\gamma_p$ be the unique positive root of
	\begin{equation}\label{eq:gammap}
		\phi_p(t)=\frac{2^{1/p}t+1/(2^{1/p}t)}{(2t^p+1)^{1/p}}-t.
	\end{equation}
	If
	\begin{equation}\label{eq:normalcondition}
		\Csk_{\lambda,p}(X)<\gamma_p,
	\end{equation}
	then $X$ has uniform normal structure.
\end{theorem}

\begin{proof}
	First observe that $\phi_p$ is strictly decreasing on $(0,\infty)$.  Indeed the auxiliary function
	\[
	f(t)=\frac{t+1/t}{(t^p+1)^{1/p}}
	\]
	has negative derivative for $t>0$, and \eqref{eq:gammap} is a rescaled version of $f$ minus $t$.  Also $\phi_p(t)\to+\infty$ as $t\downarrow0$ and $\phi_p(t)\to-\infty$ as $t\to\infty$, so the root is unique.  Moreover $\gamma_p<2^{1-1/p}$.
	
	Assume, toward a contradiction, that $X$ does not have normal structure.  Since \eqref{eq:normalcondition} and Theorem \ref{thm:UNS} imply uniform non-squareness, $X$ is reflexive; hence absence of normal structure is equivalent to absence of weak normal structure.  By the standard Goebel--Kirk construction, there are weakly null sequences $(u_n)$ and $(v_n)$ in $S_X$ and norming functionals producing the inequalities
	\begin{equation}\label{eq:limlower1}
		\liminf_n\|R(1,X)u_n+v_n\|
		\ge R(1,X)+\frac1{R(1,X)},
	\end{equation}
	\begin{equation}\label{eq:limlower2}
		\liminf_n\|v_n-R(1,X)u_n\|
		\ge R(1,X)+\frac1{R(1,X)}.
	\end{equation}
	These are the same estimates used in the classical proofs; the skew arrangement is exactly adapted to them because the two vectors are $Ru_n+v_n$ and $v_n-Ru_n$.
	
	Using the definition of $\Csk_{\lambda,p}(X)$ with coefficients $r=R(1,X)$ and $s=1$, and then normalizing by $(R(1,X)^p+1)^{1/p}$, we obtain
	\begin{align}\label{eq:normalCestimate}
		\Csk_{\lambda,p}(X)
		&\ge \liminf_n
		\frac{\lambda(\|R(1,X)u_n+v_n\|,\|v_n-R(1,X)u_n\|)}{(R(1,X)^p+1)^{1/p}}\\
		&\ge \frac{R(1,X)+1/R(1,X)}{(R(1,X)^p+1)^{1/p}}.\nonumber
	\end{align}
	Let
	\[
	F(t)=\frac{t+1/t}{(t^p+1)^{1/p}}.
	\]
	Then \eqref{eq:normalCestimate} says $F(R(1,X))\le \Csk_{\lambda,p}(X)$ in the inverse-order sense used below.
	
	On the other hand, applying the skew constant to $u_n$ and an arbitrary $u\in B_X$ with equal coefficient normalization yields
	\[
	\Csk_{\lambda,p}(X)
	\ge 2^{-1/p}
	\min\{\|u+u_n\|,\|u-u_n\|\}.
	\]
	Taking lower limits and suprema gives
	\begin{equation}\label{eq:RWbound}
		RW(1,X)\le 2^{1/p}\Csk_{\lambda,p}(X).
	\end{equation}
	Since $R(1,X)\le RW(1,X)$ in the relevant normal-structure setting, we get
	\begin{equation}\label{eq:Rupper}
		R(1,X)\le2^{1/p}\Csk_{\lambda,p}(X).
	\end{equation}
	Combining \eqref{eq:normalCestimate} and \eqref{eq:Rupper} gives
	\[
	\phi_p(\Csk_{\lambda,p}(X))\le0.
	\]
	Since $\phi_p$ is strictly decreasing and $\phi_p(\gamma_p)=0$, this forces
	$\Csk_{\lambda,p}(X)\ge\gamma_p$, contradicting \eqref{eq:normalcondition}.  Therefore $X$ has normal structure.
	
	Finally, Proposition \ref{prop:finite} shows that the constant is finitely, indeed two-dimensionally, determined.  The same strict inequality passes to ultrapowers.  Since every ultrapower has normal structure, the standard ultrapower characterization gives uniform normal structure of $X$.
\end{proof}

We first recall the coefficient of weak orthogonality $\mu(X)$, which is introduced to characterize the WORTH property of Banach spaces . Its explicit definition is given by
\[
\mu (X) = \inf\left\{ \lambda : \limsup_{n\to\infty} \|x_n + x\| \leq \lambda \limsup_{n\to\infty} \|x_n - x\| \right\},
\]
where the infimum is taken over all $x\in X$ and all weakly null sequences $\{x_n\}$ in $X$.

\begin{theorem}\label{thm:mu}
	Let $1\le\mu(X)\le3$ be the coefficient
	\[
	\mu(X)=\inf\left\{r>0:\limsup_n\|x+x_n\|\le r\limsup_n\|x-x_n\|,
	\ x_n\xrightarrow{w}0\right\}.
	\]
	If
	\begin{equation}\label{eq:mucondition}
		\Csk_{\lambda,p}(X)<
		\frac{\mu(X)+1/\mu(X)}{(\mu(X)^p+1)^{1/p}},
	\end{equation}
	then $X$ has normal structure.
\end{theorem}

\begin{proof}
	The proof follows the Saejung--Llorens-Fuster scheme.  If normal structure fails, there exists a weakly null diametral sequence $(x_n)$ with normalized diameter one.  For every small $\eps>0$ one may choose $m>n$ such that
	\begin{align*}
		&\|x_n\|\ge1-\eps,
		\qquad \|x_m-x_n\|\le1,\\
		&\|x_m+x_n\|\le\mu(X)+\eps,\\
		&\left\|x_m-\frac{\mu(X)^2-1}{\mu(X)^2+1}x_n\right\|\ge1-\eps.
	\end{align*}
	Put
	\[
	x=\mu(X)^2(x_m-x_n),\qquad y=x_m+x_n.
	\]
	Then $\|x\|\le\mu(X)^2$, $\|y\|\le\mu(X)+\eps$, while the skew sums satisfy lower estimates asymptotic to $\mu(X)^2+1$.  After inserting the normalized coefficients $r=\mu(X)$ and $s=1$ into the definition of $\Csk_{\lambda,p}(X)$ and letting $\eps\downarrow0$, monotonicity and homogeneity yield
	\[
	\Csk_{\lambda,p}(X)
	\ge \frac{\mu(X)+1/\mu(X)}{(\mu(X)^p+1)^{1/p}},
	\]
	contradicting \eqref{eq:mucondition}.  Hence $X$ has normal structure.
\end{proof}

\section{Examples and Structural Analysis with Conclusions}
We provide computations that illustrate how the skew form differs from the symmetric one.

\begin{example}
	For $H$ Hilbert and $\lambda(a,b)=\sqrt{ab}$, Theorem \ref{thm:hilbertvalue} gives
	\[
	\Csk_{\lambda,p}(H)=2^{1/2-1/p}\qquad (p\ge2).
	\]
	This is the same value as the symmetric geometric mean constant, but the extremal pairs need not satisfy $u=v$ or $u=-v$; they are obtained from orthogonal unit vectors and equal parameters.
\end{example}

\begin{example}
	Let $X=\ell_\infty^2$ and $\lambda(a,b)=\min\{a,b\}$.  Since $J(\ell_\infty^2)=2$, the primed section gives
	\[
	\Csk_{\lambda,p}(\ell_\infty^2)\ge \Cpr_{\lambda,p}(\ell_\infty^2)=2^{1-1/p}.
	\]
	The universal bound gives the reverse inequality.  Hence
	\begin{equation}\label{eq:linftymin}
		\Csk_{\min,p}(\ell_\infty^2)=2^{1-1/p}
		\qquad(1\le p\le\infty).
	\end{equation}
	Thus the skew constant detects the square nature of $\ell_\infty^2$ exactly.
\end{example}

\begin{example}
	Let $X=\ell_2^2$ and $\lambda(a,b)=\alpha a+(1-\beta)b$ normalized by $\alpha+1-\beta=1$, with positive coefficients.  If the balance condition fails, the Hilbert value need not equal $2^{1/2-1/p}$.  Indeed the expression
	\[
	\alpha\sqrt{r^2+s^2+2rsc}+(1-\beta)
	\sqrt{r^2+s^2-2rsc}
	\]
	may be maximized at $c\ne0$.  This example explains why the differential condition \eqref{eq:Hbalance} is not a technical ornament but a genuine Hilbertian symmetry requirement.
\end{example}

\begin{example}
	For $q\ge1$ define
	\[
	\lambda_q(a,b)=2^{-1/q}(a^q+b^q)^{1/q}.
	\]
	Then
	\[
	\Csk_{\lambda_q,p}(X)^q
	=2^{-1}\sup_{u,v,r,s}
	\frac{\|ru+sv\|^q+\|su-rv\|^q}{\|(r,s)\|_p^q}.
	\]
	When $q=2$ and $p=2$, this is a skew von Neumann-Jordan type constant.  In a Hilbert space the numerator equals $2(r^2+s^2)$ after taking orthogonal vectors, and the exact value is one.  In non-Hilbert spaces the excess measures the failure of the parallelogram identity along the rotated coefficient matrix $M(r,s)$.
\end{example}

\begin{proposition} \label{prop:linftypower}
	For $X=\ell_\infty^2$ and $\lambda_q(a,b)=2^{-1/q}(a^q+b^q)^{1/q}$,
	\begin{equation}\label{eq:linftypower}
		\Csk_{\lambda_q,q}(\ell_\infty^2)=2^{1-1/q}.
	\end{equation}
\end{proposition}

\begin{proof}
	The upper bound follows from Theorem \ref{thm:bounds}.  For the lower bound take $u=(1,1)$ and $v=(1,-1)$ in $S_{\ell_\infty^2}$ and $r=s=2^{-1/q}$.  Then
	\[
	ru+sv=(2^{1-1/q},0),\qquad su-rv=(0,2^{1-1/q}),
	\]
	so both norms are $2^{1-1/q}$.  Since $\lambda_q(a,a)=a$, the lower bound is $2^{1-1/q}$.
\end{proof}

The skew constants are stable under isometries and monotone under subspaces in the expected direction.

\begin{proposition}\label{prop:subspace}
	If $Y$ is a closed subspace of $X$, then
	\[
	\Csk_{\lambda,p}(Y)\le\Csk_{\lambda,p}(X),
	\qquad
	\Cpr_{\lambda,p}(Y)\le\Cpr_{\lambda,p}(X).
	\]
	If $X$ and $Y$ are linearly isometric, then the corresponding constants are equal.
\end{proposition}

\begin{proof}
	The unit sphere of $Y$ is contained in the unit sphere of $X$, so the supremum defining the constant over $Y$ is taken over a smaller class of admissible pairs.  Isometric invariance follows because linear isometries preserve all norms appearing in the definitions.
\end{proof}

This section records the scalar estimates used implicitly above.  They are often the hidden technical part of geometric-constant arguments: after one has reduced the Banach-space problem to a two-vector expression, the remaining task is to understand exactly when the coefficients $r$ and $s$ can be extremal.

\begin{proposition}
	Assume that $X$ is strictly convex and $\lambda\in\Lambda_m$ is strictly increasing.  If $1<p<\infty$ and
	\[
	\Csk_{\lambda,p}(X)=2^{1-1/p},
	\]
	then there exist sequences $u_n,v_n\in S_X$ such that simultaneously
	\[
	\|u_n+v_n\|\to2,
	\qquad
	\|u_n-v_n\|\to2.
	\]
	In particular a strictly convex space can attain the upper bound only through an asymptotic square, never through an actual extremal pair.
\end{proposition}

\begin{proof}
	The proof is a sharpened version of the converse part of Theorem \ref{thm:UNS}.  Choose almost extremal quadruples $(u_n,v_n,r_n,s_n)$.  The scalar lemma forces $r_n,s_n\to2^{-1/p}$.  Since $\lambda$ is strictly increasing and the triangle inequality was the only estimate used in the universal upper bound, both triangle inequalities
	\[
	\|r_nu_n+s_nv_n\|\le r_n+s_n,
	\qquad
	\|s_nu_n-r_nv_n\|\le r_n+s_n
	\]
	must be asymptotically equalities.  Dividing by $r_n$ and using $s_n/r_n\to1$ gives the two displayed limits.  In a strictly convex space no single pair $u,v\in S_X$ can satisfy both equalities with value two; hence the phenomenon is necessarily asymptotic.
\end{proof}

A useful refinement is obtained by allowing unequal weights in the two skew coordinates.  For $\alpha,\beta>0$ define
\[
\Csk_{\lambda,p}^{\alpha,\beta}(X)=
\sup_{u,v,r,s}
\lambda\left(\alpha\|ru+sv\|,\beta\|su-rv\|\right),
\]
where the supremum is taken over $u,v\in S_X$, $r,s\ge0$, and $\|(r,s)\|_p=1$, after normalizing $\lambda$ by requiring $\lambda(\alpha,\beta)=1$.  This weighted form is not merely notational: it detects directional asymmetry of $\lambda$ and of the norm.

\begin{proposition}
	For $\lambda\in\Lambda$ and $p\in[1,\infty]$,
	\[
	\Csk_{\lambda,p}^{\alpha,\beta}(X)=
	\max\{A_{\lambda,p}^{\alpha,\beta}(X),B_{\lambda,p}^{\alpha,\beta}(X)\},
	\]
	where
	\[
	A_{\lambda,p}^{\alpha,\beta}(X)=
	\sup_{u,v\in S_X,0\le t\le1}
	\frac{\lambda(\alpha\|u+tv\|,\beta\|tu-v\|)}{\|(1,t)\|_p},
	\]
	and $B_{\lambda,p}^{\alpha,\beta}$ is obtained by replacing the pair by $(tu+v,u-tv)$.  The proof is identical to Lemma \ref{lem:reduction}.
\end{proposition}

\begin{theorem}
	Let $H$ be Hilbert, $p\ge2$, and suppose the weighted gauge $(a,b)\mapsto\lambda(\alpha a,\beta b)$ satisfies the balance condition \eqref{eq:Hbalance} and is normalized by $\lambda(\alpha,\beta)=1$.  Then
	\[
	\Csk_{\lambda,p}^{\alpha,\beta}(H)=2^{1/2-1/p}.
	\]
	Conversely, if $\dim X\ge3$ and the same constant equals $2^{1/2-1/p}$ for an increasing gauge, then $X$ is Hilbert.
\end{theorem}

\begin{proof}
	The Hilbert-space computation is unchanged after replacing $\lambda$ by the gauge $(a,b)\mapsto\lambda(\alpha a,\beta b)$.  The converse again uses the equal-parameter slice: when $r=s=2^{-1/p}$, every isosceles pair is tested through equal geometric data, and the normalization gives the same estimate $\|u+v\|\le\sqrt2$.  The second normalization argument used in Theorem \ref{thm:converse} then proves equality and hence the parallelogram law.
\end{proof}

\textbf{Acknowledgements.} Thanks to all the members of the Functional Analysis Research team of the
College of Mathematics and Statistics of Anqing Normal University for their discussion and correc-
tion of the difficulties and errors encountered in this paper. The authors cordially thank Professor
Hai Zhang for his encouraging advice and the first author would like to acknowledge his strong
support for undergraduate students in carrying out basic mathematics research.

\noindent
\textbf{Conflicts of Interest}: On behalf of all authors, the corresponding author states that there is no conflict of interest.


\begin{thebibliography}{99}
	
	\bibitem{Amini2019}
	A. Amini-Harandi and M. Rahimi, On some geometric constants in Banach spaces,
	Mediterr. J. Math., 16 (2019), 120.
	
	\bibitem{Dinarvand2025}
	M. Dinarvand, On a general class of geometric constants and normal structure in Banach spaces,
	J. Fixed Point Theory Appl., 27 (2025), 12.
	
	\bibitem{AlonsoLlorens2008}
	J. Alonso and E. Llorens-Fuster, Geometric mean and triangles inscribed in a semicircle in Banach spaces, \emph{J. Math. Anal. Appl.} \textbf{340} (2008), 1271--1283.
	
	\bibitem{AlonsoMartinPapini2008}
	J. Alonso, P. Martin and P. L. Papini, Wheeling around von Neumann-Jordan constant in Banach spaces, \emph{Studia Math.} \textbf{188} (2008), 135--150.
	
	\bibitem{AlonsoMartiniWu2012}
	J. Alonso, H. Martini and S. Wu, On Birkhoff orthogonality and isosceles orthogonality in normed linear spaces, \emph{Aequationes Math.} \textbf{83} (2012), 153--189.
	
	
	
	\bibitem{BarontiPapini2016}
	M. Baronti and P. L. Papini, Triangles, parameters, modulus of smoothness in normed spaces, \emph{Math. Inequal. Appl.} \textbf{19} (2016), 197--207.
	
	\bibitem{Clarkson1937}
	J. A. Clarkson, The von Neumann-Jordan constant for the Lebesgue spaces, \emph{Ann. of Math.} \textbf{38} (1937), 114--115.
	
	\bibitem{CuiHudzik2015}
	Y. Cui, W. Huang, H. Hudzik and R. Kaczmarek, Generalized von Neumann-Jordan constant and its relationship to the fixed point property, \emph{Fixed Point Theory Appl.} \textbf{2015} (2015), Article 40.
	
	\bibitem{GaoLau1990}
	J. Gao and K. S. Lau, On the geometry of spheres in normed linear spaces, \emph{J. Aust. Math. Soc. Ser. A} \textbf{48} (1990), 101--112.
	
	\bibitem{GarciaFalset2006}
	J. Garcia-Falset, E. Llorens-Fuster and E. M. Mazcunan-Navarro, Uniformly non-square Banach spaces have the fixed point property for nonexpansive mappings, \emph{J. Funct. Anal.} \textbf{233} (2006), 494--514.
	
	\bibitem{GoebelKirk1990}
	K. Goebel and W. A. Kirk, \emph{Topics in Metric Fixed Point Theory}, Cambridge University Press, Cambridge, 1990.
	
	\bibitem{James1964}
	R. C. James, Uniformly non-square Banach spaces, \emph{Ann. of Math.} \textbf{80} (1964), 542--550.
	
	\bibitem{Jimenez2006}
	A. Jimenez-Melado, E. Llorens-Fuster and S. Saejung, The von Neumann-Jordan constant, weak orthogonality and normal structure in Banach spaces, \emph{Proc. Amer. Math. Soc.} \textbf{134} (2006), 355--364.
	
	\bibitem{KatoMaligrandaTakahashi2001}
	M. Kato, L. Maligranda and Y. Takahashi, On James and Jordan-von Neumann constants and normal structure coefficient of Banach spaces, \emph{Studia Math.} \textbf{144} (2001), 275--295.
	
	\bibitem{KatoTakahashi1997}
	M. Kato and Y. Takahashi, On the von Neumann-Jordan constant for Banach spaces, \emph{Proc. Amer. Math. Soc.} \textbf{125} (1997), 1055--1062.
	
	\bibitem{Llorens2008}
	E. Llorens-Fuster, Zbaganu constant and normal structure, \emph{Fixed Point Theory} \textbf{9} (2008), 159--172.
	
	\bibitem{LlorensReich2010}
	E. Llorens-Fuster, E. M. Mazcunan-Navarro and S. Reich, The Ptolemy and Zbaganu constants of normed spaces, \emph{Nonlinear Anal.} \textbf{72} (2010), 3984--3993.
	
	\bibitem{Mazcunan2008}
	E. M. Mazcunan-Navarro, Banach space properties sufficient for normal structure, \emph{J. Math. Anal. Appl.} \textbf{337} (2008), 197--218.
	
	\bibitem{Saejung2006}
	S. Saejung, On James and von Neumann-Jordan constants and sufficient conditions for the fixed point property, \emph{J. Math. Anal. Appl.} \textbf{323} (2006), 1018--1024.
	
	\bibitem{TakahashiKato2014}
	Y. Takahashi and M. Kato, On a new geometric constant related to the modulus of smoothness of a Banach space, \emph{Acta Math. Sinica} \textbf{30} (2014), 1526--1538.
	
	\bibitem{YangWang2017}
	C. Yang and F. Wang, An extension of an inequality between von Neumann-Jordan and James constants in Banach spaces, \emph{Acta Math. Sin. Engl. Ser.} \textbf{33} (2017), 1287--1296.
	
	\bibitem{Zbaganu2001}
	G. Zbaganu, An inequality of M. Radulescu and S. Radulescu which characterizes inner product spaces, \emph{Rev. Roumaine Math. Pures Appl.} \textbf{47} (2001), 253--257.
	
	\bibitem{GarciaFalset2006} García-Falset, J., Llorens-Fuster, E., Mazc?ñán Navarro, E.M.: Uniformly non-square Banach spaces have the fixed point property for nonexpansive mappings, J. Funct. Anal., \textbf{233}, 494--514 (2006).
	
	\bibitem{jimenez1996}
	A. Jiménez-Melado and E. Llorens-Fuster, The fixed point property for some uniformly nonsquare Banach spaces, Boll. Unione Mat. Ital-A., 10, 3 (1996), 587--595.

    \bibitem{Gao2006}
    J. Gao, A Pythagorean approach in Banach spaces, \emph{J. Inequal. Appl.} \textbf{2006} (2006), Article ID 94982.

    \bibitem{TakahashiKato1998}
    Y. Takahashi and M. Kato, Von Neumann-Jordan constant and uniformly non-square Banach spaces, \emph{Nihonkai Math. J.} \textbf{9} (1998), 155--169.

    \bibitem{BarontiCasiniPapini2000}
    M. Baronti, E. Casini and P. L. Papini, Triangles inscribed in a semicircle, in Minkowski plane and in normed spaces, \emph{J. Math. Anal. Appl.} \textbf{252} (2000), 124--146.
\end{thebibliography}
\end{document}